\documentclass[10pt]{amsart}

\newtheorem{cotmb}{Corollary}
\newtheorem{prop}{Proposition}
\newcommand{\cal}{\mathcal}

\newtheorem{tw}{Theorem}
\newtheorem{lemma}{Lemma}

\usepackage[mathscr]{euscript}

\vspace{0.5cm}

\title{Contact twistor spaces and almost contact metric structures}
\author{Johann Davidov, Christian L. Yankov}
\address{Institute of Mathematics and Informatics \\
Bulgarian Academy of Sciences\\ Acad. G.Bonchev St. Bl.8 \\
1113 Sofia\\ Bulgaria\\ \newline \centerline{and} \newline
University of Structural Engineering and Architecture "L.Karavelov",
175 Suhodolska St., 1373 Sofia, Bulgaria }

\email{jtd@math.bas.bg}

\address{Department of Mathematics and Computer Science\\
Eastern Connecticut State University\\
Willimantic, CT 06226 \\USA} \email{yankovc@easternct.edu}

\thanks{The first named author is partially supported by  the National Science
Fund, Ministry of Education and Science of Bulgaria under contract
DFNI-I 02/14. }

\begin{document}

\date{}

\maketitle

\noindent {\it Abstract}. The notions of a twistor space of a
contact manifold and a contact connection on such a manifold have
been introduced by L. Vezzoni as extensions of the corresponding
notions in the case of a symplectic manifold. Given a contact
connection on a contact manifold one can define an almost
$CR$-structure on its twistor space and Vezzoni has found the
integrability condition for this structure. In the present paper
it is observed that the $CR$-structure is induced by an almost
contact metric structure. The main goal of the paper is to obtain
necessary and sufficient conditions for normality of this
structure in terms of the curvature of the given contact
connection. Illustrating examples are discussed at the end of the
paper.



\thispagestyle{empty}

\section{Introduction}

Let $M$ be a contact manifold with contact form $\alpha$ and contact
distribution ${\cal D}$. The restriction of $d\alpha$ to ${\cal D}$
is a symplectic form and L. Vezzoni \cite{Vezz} has defined the
contact twistor space of $M$ as the bundle over $M$ whose fiber at
any point $p\in M$ consists of all complex structures on ${\cal
D}_p$, the fibre of ${\cal D}$ at $p$, compatible with the
symplectic form $d\alpha|{\cal D}$. He has also introduced the
notion of a contact connection on $M$ and showed that, given such a
connection $\nabla$, one can define a natural almost $CR$-structure
on the contact twistor space of $M$ in a way that resembles the
standard twistor construction. Vezzoni has found the integrability
condition for this almost $CR$-structure in terms of the curvature
of the connection $\nabla$.

In this note we observe that the $CR$-structure considered in
\cite{Vezz} is induced by an almost contact metric structure on
the contact twistor space of $M$. As usual in twistor theory, one
can define one more almost contact metric structure and the main
purpose of this paper is to discuss the normality of these almost
contact metric structures. Recall that normality is an important
property of an almost contact manifold $N$, which means that the
product manifold $N\times S^1$ is a complex manifold with the
complex structure induced by the almost contact one (cf., for
example, \cite{Blair}). As one can expect, the integrability
condition for the first almost contact metric structure can be
expressed in terms of the curvature of the connection $\nabla$,
while the second one is never normal. Considering the  two induced
$CR$-structures, we reprove the Vezzoni integrability result for
the first one and show that the second $CR$-structure is never
integrable. Examples illustrating the obtained results are
discussed in the last section of the paper.

\section{Preliminaries}

\subsection{Compatible complex structures on a symplectic vector
space} Let ${\cal D}$ be a $2n$-dimensional real vector space
endowed with a non-degenerate skew-symmetric $2$-form $\omega$.
Denote by ${\cal Z}={\cal Z}({\cal D},\omega)$ the set of all
complex structures $J$ on the vector space ${\cal D}$ compatible
with the form $\omega$, i.e. such that $\omega(Jx,Jy)=\omega(x,y)$
and $\omega(x,Jx)>0$ if $x\neq 0$. This set is a submanifold of the
symplectic Lie algebra
$$
sp(\omega)=\{S\in End({\cal
D}):~\omega(Sx,y)+\omega(x,Sy)=0,~~x,y\in{\cal D}\}.
$$
The tangent space $T_J{\cal Z}$ (considered as a subspace of
$sp(\omega)$) is
$$
T_J{\cal Z}=\{V\in End({\cal
D}):~VJ+JV=0,~\omega(Vx,y)+\omega(x,Vy)=0\}.
$$
The smooth manifold ${\cal Z}$ admits a natural almost complex
structure ${\cal J}$ defined by
$$
{\cal J}V=JV~~\textrm{for}~~V\in T_J{\cal Z}.
$$

Every $J\in{\cal Z}$ is an orthogonal transformation of ${\cal D}$
with respect to the Euclidean metric $g_J(x,y)=\omega(x,Jy)$. If
$E_1,\dots,E_{2n}$ is a $g_J$-orthonormal basis such that
$JE_i=E_{i+n}$, $i=1,\dots,n$, then
$$
\omega(E_i,E_j)=\omega(E_{i+n},E_{j+n})=0,\quad
\omega(E_i,E_{j+n})=\delta_{ij}.
$$
A basis that satisfies the latter identities is called symplectic.
Vice versa, given a symplectic basis $E_1,\dots,E_{2n}$, define a
complex structure $J$ on the vector space ${\cal D}$ setting
$JE_{i}=E_{i+n}$, $JE_{i+n}=-E_{i}$, $i=1,\dots,n$. Then
$J\in{\cal Z}$ and the symplectic basis is $g_J$-orthonormal.

The metric $g_J$ induces a metric $G_J$ on the vector space
$End({\cal D})$,
$$
G_J(A,B)=Trace\{{\cal D}\ni x\to g_J(Ax,Bx)\},\quad A,B\in End({\cal
D}).
$$
Then ${\cal Z}\ni J\to G_J|T_J{\cal Z}$ is a (smooth) Riemannian
metric on the manifold ${\cal Z}$ compatible with the almost
complex structure ${\cal J}$.

Let $J\in{\cal Z}$ and let $E_{\alpha}$, $\alpha=1,\dots,2n$, be
an orthonormal basis of ${\cal D}$ with respect to the metric
$g_J$ such that $JE_i=E_{i+n}$, $i=1,\dots,n$. Define a basis
$L_{\alpha\beta}$ of the vector space $End({\cal D})$ by
$$L_{\alpha\beta}E_{\gamma}=\delta_{\alpha\gamma}E_{\beta},\quad \alpha,\beta,\gamma=1,\dots,2n.$$
This basis is orthonormal with respect to the metric $G_J$ on
$End({\cal D})$ induced by $g_J$. Set
\begin{equation}\label{Vij}
\begin{array}{c}
V_{ij}=\displaystyle{\frac{1}{2}}(L_{i+n,j}+L_{j+n,i}+L_{i,j+n}+L_{j,i+n}),\quad
i\neq j,\quad i,j=1,\dots,n,\\[6pt]
V_{ii}=\displaystyle{\frac{1}{\sqrt 2}}(L_{i+n,i}+L_{i,i+n}),\quad
i=1,\dots,n.
\end{array}
\end{equation}
These endomorphisms of ${\cal D}$ are orthonormal, lie in $T_J{\cal
Z}$ and
\begin{equation}\label{JVij}
\begin{array}{c}
{\cal
J}V_{ij}=\displaystyle{\frac{1}{2}}(L_{i+n,j+n}+L_{j+n,i+n}-L_{i,j}-L_{j,i}),\quad i\neq j\quad i,j=1,\dots,n,\\[6pt]
{\cal J}V_{ii}=\displaystyle{\frac{1}{\sqrt
2}}(L_{i+n,i+n}-L_{i,i}),\quad i=1,\dots,n.
\end{array}
\end{equation}
Thus $\{V_{ij},{\cal J}V_{ij}:1\leq i\leq j\leq n\}$ is a
$G_J$-orthonormal basis of $T_J{\cal Z}$.

\smallskip

Denote by $symm(\omega)$ the set of endomorphisms $S$ of ${\cal D}$
that are $\omega$-symmetric, i.e. $\omega(Sx,y)=\omega(x,Sy)$. Then
$$
End({\cal D})=sp(\omega)\oplus symm(\omega).
$$
Given $A\in End({\cal D})$ define an endomorphism $A^{\ast}$ of
${\cal D}$ by
$$
\omega(A^{\ast}x,y)=\omega(x,Ay).
$$
Then $\check{A}=\frac{1}{2}(A-A^{\ast})$ and
$\hat{A}=\frac{1}{2}(A+A^{\ast})$ are the projections of $A$ onto
$sp(\omega)$ and $symm(\omega)$, respectively.

 For every $J\in{\cal Z}$, we also have the direct sum decomposition
$$
sp(\omega)=T_J{\cal Z}\oplus \{S\in sp(\omega):~JS-SJ=0\}.
$$
The projection of an endomorphism $U\in sp(\omega)$ to $T_J{\cal Z}$
is given by $pr_J(U)=\frac{1}{2}(U+JUJ)$. Thus
$$
pr_J(A)=\frac{1}{2}(\check{A}+J\check{A}J).
$$
is the projection of $A\in End({\cal D})$ to $T_J{\cal Z}$ with
respect to the decomposition
$$
End({\cal D})=T_J{\cal Z}\oplus \{S\in sp(\omega):~JS-SJ=0\}\oplus
symm(\omega).
$$

\smallskip

Let $U\in T_J{\cal Z}$ and let $V$ be a vector field on a
neighbourhood of $J$.  Take any smooth function $sp(\omega)\to
sp(\omega)$ that coincides with V on a neighbourhood of $J$ and
denote this function again by $V$. It follows from the Koszul
formula that the Levi-Civita connection $D$ of the Riemannian
manifold $({\cal Z},G)$ is given by
\begin{equation}\label{D}
(D_{U}V)_J=pr_J\big(V'(J)(U)\big)
\end{equation}
where $V'(J)\in End(sp(\omega))$ is the derivative of the function
$V:sp(\omega)\to sp(\omega)$. Then
$$
(D_{U}{\cal
J}V)_J=pr_J\big(UV_J+J(V'(J))(U)\big)=Jpr_J\big(V'(J)(U)\big)={\cal
J}(D_{U}V)_J
$$
Therefore $(G,{\cal J})$ is a K\"ahler structure on ${\cal Z}$.

The Lie group $Sp(\omega)$ of linear transformations $Q$ of ${\cal
D}$ that preserve $\omega$, $\omega(Qx,Qy)\\=\omega(x,y)$, acts
transitively on ${\cal Z}$ by conjugation. Indeed, if $J,J'\in{\cal
Z}$, take symplectic bases $E_{\alpha}$ and $E'_{\alpha}$,
$\alpha=1,\dots,2n$, determined by $J$ and $J'$, respectively. Then,
if $Q$ is the linear transformation of ${\cal D}$ defined by
$QE_{\alpha}=E'_{\alpha}$, we have $Q\in Sp(\omega)$ and
$J'=QJQ^{-1}$. Denote the isotropy subgroup of $Sp(\omega)$ at a
point $J^0\in {\cal Z}$ by $U(J^0)$. Thus ${\cal Z}$ is the
homogeneous space $Sp(\omega)/U(J^0)$. Note also that the complex
structure ${\cal J}$ and the metric $G$ on ${\cal Z}$ are invariant
under the action of $Sp(\omega)$.

The Lie algebras of the groups $Sp(\omega)$ and $U(J_0)$ are
$sp(\omega)$ and $\frak u=\{S\in sp(\omega):~SJ_0-J_0S=0\}$. Set
$\frak m=\{S\in sp(\omega):~SJ_0+J_0S=0\}$. Then
$$
sp(\omega)=\frak u\oplus\frak m,\quad [\frak u,\frak m]\subset \frak
m,\quad [\frak m,\frak m]\subset\frak u.
$$
Fix a symplectic basis $E^0_{\alpha}$, $\alpha=1,\dots,2n$, such
that $J^0E^0_i=E^0_{i+n}$, $J^0E^0_{i+n}=-E^0_{i}$, $i=1,\dots,n$.
Consider the isomorphism of ${\cal D}$ with ${\Bbb R}^{2n}$ that
sends $E^0_{\alpha}$ to the standard basis of ${\Bbb R}^{2n}$ .
Then $Sp(\omega)\cong Sp(2n,{\Bbb R})$ and $U(J^0)\cong U(n)$.
Thus ${\cal Z}\cong Sp(\omega)/U(J^0)\cong Sp(2n,{\Bbb R})/U(n)$
is a symmetric space. Recall that $Sp(2n,{\Bbb R})/U(n)$ is
diffeomorphic to the Siegel upper half plane ${\Bbb S}_n$, which
is the set of complex symmetric $n\times n$-matrices $Z=X+iY$ with
positive definite imaginary part $Y$ (see, for example,
\cite{McDS} or \cite{V86}). Indeed, let
$$
\psi=\left(\begin{array}{cc}
           A & B \\
           C & D \\
           \end{array}
                      \right)
$$
be the matrix representation of a transformation $\psi\in
Sp(2n,{\Bbb R})$ with respect to the standard basis of ${\Bbb
R}^{2n}$. Then
$$
\psi\cdot\! Z=(AZ+B)(CZ+D)^{-1}
$$
defines a transitive action of $Sp(2n,{\Bbb R})$ on ${\Bbb S}_n$.
The isotropy subgroup at the matrix $iI\in {\Bbb S}_n$, $I$ being
the identity $n\times n$-matrix,  is $U(n)$. Thus $Sp(2n,{\Bbb
R})/U(n)\cong {\Bbb S}_n$. Let $J(Z)$ be the complex structure
corresponding to $Z=X+iY\in {\Bbb S}_n$ under the composite
diffeomorphism ${\Bbb S}_n\cong Sp(2n,{\Bbb R})/U(n)\cong{\cal Z}$.
Denote by $Y^{1/2}$ the principal square root of the symmetric
positive definite matrix $Y$. Then the matrix
$$
\psi=\left(\begin{array}{cc}
           Y^{1/2} & XY^{-1/2} \\
           0 & Y^{-1/2} \\
           \end{array}
                      \right)
$$
represents a transformation in $Sp(2n,{\Bbb R})$ (smoothly
depending on $(X,Y)$) such that $\psi\cdot\!(iI)=Z$. Hence
$$
J(Z)=\psi \left(\begin{array}{cc}
           0 & I\\
           -I & 0 \\
           \end{array}
                      \right)\psi^{-1}
=\left(\begin{array}{cc}
           -XY & Y+XY^{-1}X \\
           -Y^{-1} & Y^{-1}X \\
           \end{array}
                      \right).
$$
It is easy to check by means of the latter formula that the
diffeomeorphism $Z\to J(Z)$ is holomorphic, ${\Bbb S}_n$ being
considered with its natural complex structure as an open subset of
the vector space of symmetric complex matrices. The manifold ${\Bbb
S}_n$ admits a $Sp(2n,{\Bbb R})$-invariant K\"ahler metric $H$
introduced and studied by Siegel (see, for example, \cite{S}). For
$W=U+iV\in T_{Z}{\Bbb S}_n$, it is defined by
$$
H(W,W)=Trace(Y^{-1}UY^{-1}U+Y^{-1}VY^{-1}V)
$$
One can easily check that the biholomorphism $Z\to J(Z)$ sends the
metric $G$ on ${\cal Z}$ to the metric $2H$ on ${\Bbb S}_n$.

\subsection{$Sp(\omega)$-decomposition of curvature tensors}

Let $(M,\omega)$ be a symplectic manifold and $\nabla$ a linear
torsion-free connection on $TM$ for which $\nabla\omega=0$. Such a
connection always exists \cite{Lb} (see also \cite{V85}). Let $R$ be
the curvature tensor of $\nabla$.

\smallskip

\noindent {\bf Convention}. We adopt the following definition for
the curvature tensor
$R(X,Y)=\nabla_{[X,Y]}-[\nabla_{X},\nabla_{Y}]$. This differs by a
sign from the definition used in \cite{V85,Vezz}.

\smallskip

Set
$$
R(X,Y,Z,U)=\omega(R(X,Y)Z,U),\quad X,Y,Z,U\in TM.
$$

It has been observed in \cite{V85} that this covariant $4$-tensor
has the following properties:
\begin{itemize}
\item[$(i)$] \quad $R(X,Y,Z,U)=-R(Y,X,Z,U)$;
\item[$(ii)$]\quad $R(X,Y,Z,U)=R(X,Y,U,Z)$;
\item[$(iii)$]\quad $R(X,Y,Z,U)+ R(Y,Z,X,U)+ R(Z,X,Y,U)=0$.
\end{itemize}

\smallskip

Now consider the space ${\mathscr S}$ of covariant $4$-tensors on
a symplectic vector space $({\cal D},\omega)$ having these
properties. The group $Sp(\omega)$ acts on the space ${\mathscr
S}$ in a natural way. The irreducible decomposition of ${\mathscr
S}$ under the action of $Sp(\omega)$ has been found in \cite{V85}.
To describe this decomposition we introduce the Ricci tensor of a
tensor $R\in {\mathscr S}$ in the usual way. Let $R(X,Y)Z$ be the
$(1,3)$-tensor defined by
$$
\omega(R(X,Y)Z,U)=R(X,Y,Z,U).
$$
Then the Ricci tensor of $R$ is defined as
$$
\sigma_R(X,Y)=Trace\{Z\to R(X,Z)Y\}.
$$
Thus, if $E_1,\dots,E_{2n}$ is a symplectic basis,
$$
\sigma_R(X,Y)=\sum_{i=1}^n [R(X,E_i,Y,E_{i+n})-R(X,E_{i+n},Y,E_i)].
$$
It has been shown in \cite{V85} that the Ricci tensor $\sigma_R$
is symmetric. Let ${\mathscr S}^0$ be the subspace of tensors in
${\mathscr S}$ with vanishing Ricci tensor. Denote by ${\mathscr
S}^r$ the subspace of ${\mathscr S}$ consisting of all tensors in
${\mathscr S}$ of the form
\begin{equation}\label{red}
\begin{array}{r}
R(X,Y,Z,U)=\displaystyle{\frac{1}{2n+2}}\big[-\omega(X,Z)P(Y,U)+\omega(Y,U)P(X,Z)\\[8pt]
\hspace{3cm}-\omega(X,U)P(Y,Z)+\omega(Y,Z)P(X,U)\\[6pt]
-2\omega(X,Y)P(Z,U)\big]
\end{array}
\end{equation}
where $P$ is a symmetric covariant $2$-tensor.
\begin{prop} {\rm{(\cite{V85})}}\quad $(i)$ If $n=1$, then ${\mathscr S}$ is an irreducible
$Sp(\omega)$-module and ${\mathscr S}={\mathscr S}^r$;

\noindent $(ii)$ If $n>1$, we have the following
$Sp(\omega)$-irreducible decomposition
$$
{\mathscr S}={\mathscr S}^0\oplus {\mathscr S}^r.
$$
\end{prop}
The projection  to ${\mathscr S}^r$ of a tensor $R\in {\mathscr S}$
is given by the right-hand side of (\ref{red}) with $P=\sigma_R$
(\cite{V85}).

\smallskip

\noindent {\it Definition}. As in \cite{Vezz}, we shall say that a
covariant $4$-tensor $R$ in ${\mathscr S}$ is of {\it Ricci type} if
$R\in {\mathscr S}^r$ (i.e. if $R$ is a reducible symplectic
curvature tensor in the terminology of \cite{V85}).

\subsection{Contact connections} Let $M$ be a $(2n+1)$-dimensional contact manifold with
contact form $\alpha$. We denote the contact distribution
$Ker\,\alpha$ by ${\cal D}$ and the Reeb vector field by $\xi$.

Following \cite {Vezz}, a linear connection $\nabla$ on $TM$ will
be called contact, if for every $X\in TM$ and every three sections
$Y,Y_1,Y_2$ of ${\mathcal D}$ and every two sections $X_1, X_2$ of $TM$
\begin{equation}\label{ncd}
\nabla_{X}Y \mbox{ is a section of}~ {\cal D},
\end{equation}
\begin{equation}\label{lie}
\nabla_{\xi}Y=[\xi,Y],
\end{equation}
\begin{equation}\label{nxi}
\nabla_{X}\xi=0,
\end{equation}
\begin{equation}\label{nda}
(\nabla_{Y}d\alpha)(Y_1,Y_2)=0,
\end{equation}
\begin{equation}\label{tor}
[X_1,X_2]=\nabla_{X_1}X_2-\nabla_{X_2}X_1-d\alpha(X_1,X_2)\xi.
\end{equation}

The last identity implies
$\nabla_{Y_1}Y_2-\nabla_{Y_2}Y_1=[Y_1,Y_2]-\alpha([Y_1,Y_2])\xi=$
the projection of $[Y_1,Y_2]$ to ${\mathcal D}$ with respect to the
decomposition $TM={\mathcal D}\oplus {\mathbb R}\xi$. Note also that
$[\xi,Y]=\nabla_{\xi}Y\in{\cal D}$, so $\alpha([\xi,Y])=0$.

\begin{prop} {\rm{(\cite{Vezz})}} Every contact manifold admits a
contact connection. The set of all contact connections on
$(M,\alpha)$ is an affine space modeled on the space of symmetric
covariant $3$-tensors on the contact distribution ${\cal D}$.
\end{prop}

Henceforward $\nabla$ will denote a contact connection on $M$.

\smallskip

It is convenient to set $\omega=d\alpha$. Then
\begin{equation}\label{om-xi}
\omega(\xi,.)=0,
\end{equation}
while $\omega$ is a symplectic form on ${\cal D}$.
\begin{lemma}\label{no}
For every contact connection $\nabla$
$$
\nabla\omega=0.
$$
\end{lemma}
\begin{proof}
Let $X,Y,Z$ be vector fields on $M$. If $X,Y,Z$ are sections of
${\cal D}$, then $(\nabla_{X}\omega)(Y,Z)=0$ by (\ref{nda}).
Moreover, in view of (\ref{lie}),
$$
\begin{array}{c}
(\nabla_{\xi}\omega)(Y,Z)=\xi(d\alpha(Y,Z))-d\alpha(\nabla_{\xi}Y,Z)-d\alpha(Y,\nabla_{\xi}Z)\\[6pt]
=\xi(d\alpha(Y,Z))-d\alpha ([\xi,Y],Z)-d\alpha (Y,[\xi,Z])
=d(d\alpha)(\xi,Y,Z)=0.
\end{array}
$$
We also have $(\nabla_{X}\omega)(\xi,Z)=0$ for every $X,Z\in TM$
since $\omega(\xi,.)=0$ and $\nabla_{X}\xi=0$.
\end{proof}

\begin{lemma}\label{Bi}
The curvature tensor $R$ of a contact connection satisfies the
following identities
\begin{itemize}
\item[$(i)$] \quad $\omega(R(X,Y)Z,U)=\omega(R(X,Y)U,Z)$,\quad $X,Y,Z,U\in TM$;
\item[$(ii)$]\quad $R(X,Y)Z+R(Y,Z)X+R(Z,Y)X=0$ \rm{(the Bianchi
identity)}.
\end{itemize}
\end{lemma}
\begin{proof}
Since $\nabla\omega=0$, we have
$$
\begin{array}{c}
\omega(\nabla_{X}\nabla_{Y}Z,U)=XY(\omega(Z,U))-X(\omega(Z,\nabla_{Y}U))
-Y(\omega(Z,\nabla_{X}U))\\[6pt]+\omega(Z,\nabla_{Y}\nabla_{X}U)\\[6pt]
\mbox{ and }\\[6pt]
\omega(\nabla_{[X,Y]}Z,U)=[X,Y](\omega(Z,U))-\omega(Z,\nabla_{[X,Y]}U).
\end{array}
$$
It follows that $\omega(R(X,Y)Z,U)=-\omega(Z,R(X,Y)U)$. This proves
$(i)$.

To prove $(ii)$ we first note that $R(\cdot\,,\cdot)\xi=0$ since
$\nabla\xi=0$, and that if $Z\in{\cal D}$, then
$R(\cdot\,,\cdot)Z\in{\cal D}$  since $\nabla$ preserves the
bundle ${\cal D}$. Thus, to show the Binachi identity, it is
enough to consider the cases when $X,Y,Z\in{\cal D}$ and
$X,Z\in{\cal D}, Y=\xi$. In the first case the Bianchi identity
has been proved in \cite[Lemma 2.6]{Vezz}. In the second case, we
have
$$
\begin{array}{c}
R(X,\xi)Z+R(\xi,Z)X=-\nabla_{X}\nabla_{\xi}Z+\nabla_{\xi}(\nabla_{X}Z-\nabla_{Z}X)
+\nabla_{Z}\nabla_{\xi}X\\[6pt]
\hspace{4cm}+\nabla_{[X,\xi]}Z+\nabla_{[\xi,Z]}X\\[6pt]
=-\nabla_{X}[\xi,Z]+\nabla_{\xi}[X,Z]_{\cal
D}-\nabla_{Z}[X,\xi]+\nabla_{[X,\xi]}Z+\nabla_{[\xi,Z]}X-\nabla_{[X,Z]}\xi\\[6pt]
=(-[X,[\xi,Z]]+[\xi,[X,Z]]-[Z,[X,\xi]])_{\cal
D}-\xi(\alpha([X,Z]))\xi\\[6pt]
=-\xi(\alpha([X,Z]))\xi.
\end{array}
$$
where the subscript ${\cal D}$ means "the projection to ${\cal D}$".
It follows that $R(X,\xi)Z+R(\xi,Z)X=0$ since the left-hand side of
the identity above lies in ${\cal D}$.
\end{proof}

\section{Almost contact metric structures on contact twistor spaces}

Let $M$ be a contact manifold with contact form $\alpha$, contact
distribution ${\cal D}$ and Reeb field $\xi$, $dim\,M=2n+1$. Set
$\omega=d\alpha$ as above.

Following \cite{Vezz} we define the contact twistor space of
$(M,\alpha)$ as the bundle ${\cal C}\to M$ whose fibre at every
point $p\in M$ is ${\cal Z}({\cal D}_p,\omega_p)$, the space of
complex structures on the vector space ${\cal D}_p$ compatible with
the symplectic form $\omega_p|{\cal D}_p=d\alpha|{\cal D}_p$.

The total space ${\cal C}$ is a submanifold of $End({\cal D})$. We
imbed $End({\cal D})$ into $End(TM)$ setting $A\xi=0$ for every
$A\in End({\cal D})$, and shall consider ${\cal C}\to M$ as a
subbundle of the bundle $\pi:End(TM)\to M$.

\smallskip

\noindent {\it Remark 1}. According to this convention, every
$J\in{\cal C}$ will be considered as an endomorphism of
$T_{\pi(J)}M$ such that
\begin{equation}\label{J2}
J^2X=-X+\alpha(X)\xi,\quad X\in T_{\pi(J)}M,
\end{equation}
\begin{equation}\label{om-J}
\omega(JX,JY)=\omega(X,Y), \quad X,Y\in T_{\pi(J)}M,
\end{equation}
\begin{equation}
\omega(Z,JZ)>0~\mbox{ for }~ Z\in {\cal D}_{\pi(J)},\, Z\neq 0.
\end{equation}

\smallskip

Suppose we a given a contact connection $\nabla$ on $TM$.

\smallskip

The connection $\nabla$  induces a connection on the vector
bundle $End(TM)$ which will also  be denoted by $\nabla$.

\smallskip

\noindent {\it Remark 2}. Let $S$ be a section of the bundle ${\cal
C}\to M$. Denote the extension of $S$ to a section of $End(TM)$ by
$\bar S$, so $\bar S_p|{\cal D}_p=S_p$, $\bar S_p(\xi_p)=0$, $p\in
M$. Then $(\nabla_{X}\bar S)(Z)=(\nabla_{X}S)(Z)$ for $X\in T_pM$,
$Z\in {\cal D}_p$ since $\nabla$ preserves the bundle ${\cal D}$.
Also $(\nabla_{X}\bar S)(\xi)=0$ since $\nabla\xi=0$. Thus, the
extension of $\nabla_{X}S$ is $\nabla_{X}\bar S$.

\smallskip

Let ${\cal H}$ be the horizontal subbundle of $TEnd(TM)$ defined by
means of the connection $\nabla$ on $End(TM)$.

\medskip

\noindent {\bf Notation}.  Let $J\in {\cal C}$ and $p=\pi(J)$.
Take a basis $e_1,\dots,e_n, e_{n+1}=Je_1,\dots,e_{2n}=Je_n$ of
${\cal D}_p=Im\,J$ that is orthonormal with respect to the metric
$g_J(u,v)=\omega(u,Jv)$ on $D_p$. For this basis
$\omega(e_i,e_j)=\omega(e_{i+n},e_{j+n})=0$, $\omega
(e_i,e_{j+n})=\delta_{ij}$, $i,j=1,\dots.,n$. Since by
Lemma~\ref{no} $\omega$ is $\nabla$-parallel, there exists a frame
of vector fields $E_1,\dots,E_{2n}$ in a (geodesically convex)
neighbourhood of $p$ such that
$$
\begin{array}{c}
E_{r}(p)=e_{r},\quad \nabla E_{r}|_p=0,\quad
r=1,\dots,2n,\\[6pt]
\omega(E_i,E_j)=\omega(E_{i+n},E_{j+n})=0,\quad \omega
(E_i,E_{j+n})=\delta_{ij},\quad i,j=1,\dots.,n.
\end{array}
$$

Define a section $S$ of $End(TM)$ by
$$
SE_i=E_{i+n},\quad SE_{i+n}=-E_{i},\quad i=1,\dots,n,\quad S\xi=0.
$$
Then $S$ is a section of ${\cal C}$ such that
$$
S(p)=J,\quad \nabla S|_p=0.
$$

\smallskip

It follows that, for every $J\in{\cal C}$ and $X\in T_{\pi(J)}M$,
the horizontal lift $X^h_J=S_{\ast}(X)\in{\cal H}_J$ of $X$ lies in
$T_J{\cal C}$, i.e. the horizontal spaces ${\cal H}_J$, $J\in{\cal
C}$, are tangent to the manifold ${\cal C}$. Thus, if ${\cal
V}_J=Ker\,(\pi|{\cal C})_{\ast}$ is the vertical space of the bundle
${\cal C}\to M$, we have the direct sum decomposition
$$
T_J{\cal C}={\cal V}_J\oplus{\cal H}_J.
$$

\smallskip

 Let $(U,x_1,\dots,x_{2n+1})$ be a local coordinate system of $M$.
Define a frame $L_{\alpha\beta}$ of $End(TM)$ setting
$L_{\alpha\beta}E_{\gamma}=\delta_{\alpha\gamma}E_{\beta}$,
$1\leq\alpha,\beta,\gamma\leq 2n+1$. If $L\in \pi^{-1}(U)\subset
End(TM)$, we have
$$L=\sum_{\beta,\gamma=1}^{2n+1}y^{\beta\gamma}L_{\beta\gamma}$$
for  some smooth functions $y^{\beta\gamma}$. Set $\widetilde
x^{\alpha}(L)=x_{\alpha}\circ\pi (L)$. Then $(\widetilde
x^{\alpha},y^{\beta\gamma})$ is a local coordinate system of the
manifold $End(TM)$.

\smallskip

Let $[\theta_{\alpha\beta}^{\mu\nu}]$ be the connection matrix of
$\nabla$ with respect to the frame $L_{\alpha\beta}$:
$$
\nabla_X
L_{\alpha\beta}=\sum_{\mu,\nu=1}^{2n+1}\theta_{\alpha\beta}^{\mu\nu}(X)L_{\mu\nu},\quad
X\in TM.
$$
Then, for every vector field
$$
X=\sum_{\alpha=1}^{2n+1} X^{\alpha}\frac{\partial}{\partial
x_{\alpha}}
$$
on $U$, the horizontal lift $X^h$ on $\pi^{-1}(U)$ is given by
\begin{equation}\label{h-lift}
X^h=\sum_{\alpha=1}^{2n+1}(X^{\alpha}\circ\pi)\frac{\partial}{\partial
\widetilde
x^{\alpha}}-\sum_{\beta,\gamma,\mu,\nu=1}^{2n+1}y^{\beta\gamma}(\theta_{\beta\gamma}^{\mu\nu}(X)\circ\pi)
\frac{\partial}{\partial y^{\mu\nu}}.
\end{equation}

\medskip

Let $L\in End(TM)$ and $p=\pi(L)$. Then (\ref{h-lift}) implies
that under the standard identification $T_{L}End(T_pM)\cong
End(T_pM)$ we have
\begin{equation}\label{h-bra}
[X^h,Y^h]_L=[X,Y]^h_L+R(X,Y)L,
\end{equation}
where $R(X,Y)L$ is the curvature of the connection $\nabla$ on
$End(TM)$.

\smallskip

\noindent {\it Remark 3}. Note that, for $J\in{\cal C}$, the
isomorphism $T_{J}End(T_{\pi(J)}M)\cong End(T_{\pi(J)}M)$ identifies
the vertical space ${\cal V}_J$ of the bundle ${\cal C}\to M$ with
the space of endomorphisms $U$ of $T_{\pi(J)}M$ such that $U\xi=0$,
$JU+UJ=0$, $\omega(UX,Y)+\omega(X,UY)=0$, $X,Y\in T_{\pi(J)}M$.

\smallskip

\noindent {\it Remark 4}. Given $J\in{\cal C}$, denote for a moment
the extension of the endomorphism $J$ of ${\cal D}_{p}$, $p=\pi(J)$,
to an endomorphism of $T_{p}M$ by $\bar J$ ($\bar J|{\cal D}_{p}=J$,
$\bar J\xi=0$). Then, for $X,Y\in T_pM$, $R(X,Y)\bar J$ is the
extension of the endomorphism $R(X,Y)J$ of ${\cal D}_p$ since
$\nabla$ preserves ${\cal D}$ and $\nabla\xi=0$.

\smallskip

Remarks 1-4 show that the imbedding $End({\cal D})\hookrightarrow
End(TM)$ has nice properties in the context of our considerations.

\smallskip

As usual in twistor theory, we can define two endomorphisms
$\Phi_k$ of $T{\cal C}$ setting
$$
\Phi_kX^h_J=(JX)^h_J \hskip0.2cm \mbox{for}\hskip0.2cm X\in
T_{\pi(J)}M,\quad \Phi_kV=(-1)^{k+1}JV \hskip0.2cm
\mbox{for}\hskip0.2cm V\in{\cal V}_J.
$$

Clearly, $\Phi_k^3+\Phi_k=0$, $rank\,\Phi_k=2n$. Recall that an
endomorphism of the tangent bundle of a manifold with these
properties is called a partially complex structure or a
$f$-structure. Note also that $\Phi_k(\xi^h)=0$.

\smallskip

The fibre of the subbundle $Im\,\Phi_k$ of $T{\cal C}$ at a point
$J\in{\cal C}$ is the space  ${\cal V}_J\oplus \{X^h_J:~X\in{\cal
D}_{\pi(J)}\}$. Set ${\cal E}=Im\,\Phi_1 (=Im\,\Phi_2)$. Then
$({\cal E},\Phi_k|{\cal E})$ is an almost $CR$-structure on ${\cal
C}$. For $k=1$, the integrability condition for this structure has
been obtained in \cite{Vezz}.

\smallskip

For every $t>0$, we  define a Riemannian metric $G_t$ on ${\cal C}$
as follows: Let $J\in{\cal C}$ and $p=\pi(J)$. For $X,Y\in {\cal
D}_p$, we set $G_t(X^h_J,Y^h_J)=\omega(X,JY)$ and
$G_t(X^h_J,\xi^h_J)=0$. Thus
$$
G_t(X^h_J,Y^h_J)=\omega(X,JY)+\alpha(X)\alpha(Y)~~~\mbox {for
every}~~~ X,Y\in T_pM.
$$
On the vertical subspace ${\cal V}_J$ of $T_J{\cal C}$, we set
$G|{\cal V}_J=tG_J$, $t$-times the metric on the fibre through $J$.
Finally, the horizontal and vertical spaces at $J$ are declared to
be orthogonal. Then $(\Phi_k,\xi^h,G_t)$ is an almost contact metric
structure on ${\cal C}$. We refer to \cite{Blair} for general facts
about (almost)  contact metric structures.

The main purpose of this section is to find conditions on $M$
under which $(\Phi_k,\xi^h,G_t)$ is a normal structure. Recall
that any almost contact metric structure $(\varphi,\xi,g)$ on a
manifold $N$ induces an almost complex structure $K$ on the
manifold $N\times S^1$ for which $KX=\varphi X$  for $X\in TN$,
$X\perp \xi$, $K\xi=-\displaystyle{\frac{\partial}{\partial s}}\in
TS^1$ where $s$ is the local coordinate $\displaystyle{e^{2\pi
is}\to s}$ on $S^1$. The structure $(\varphi,\xi,g)$ is said to be
normal if the induced almost complex structure on $N\times S^1$ is
integrable. It is well-known that $(\varphi,\xi,g)$ is a normal
structure if and only if the tensor $N^{(1)}(X,Y)=\varphi
^{2}[X,Y]+ [\varphi X,\varphi Y]-\varphi [\varphi X,Y] -\varphi
[X,\varphi Y] +d\eta(X,Y)\xi$ vanishes, where $\eta(X)=g(X,\xi)$
(see, for example, \cite{Blair}; the definition of $d\eta$ used
here is twice the one in \cite{Blair}). For the almost contact
structure $(\Phi_k,\xi^h,G_t)$ this tensor  will be denoted by
$N^{(1)}_{k}$.

\smallskip

Let $A$ be a (local) section of $End(TM)$ with $A\xi=0$ (i.e. a
section of $End({\cal D}$)). Define a section $A^{\ast}$ of
$End({\cal D})$ by $\omega(A^{\ast}X,Y)=\omega(X,AY)$, $X,Y\in TM$,
and consider it as a section of $End(TM)$ ($A^{\ast}\xi=0$). Then
$$\check{A}=\frac{1}{2}(A-A^{\ast})$$ is an $\omega$-skew-symmetric
section of $End(TM)$, and we can define a vertical vector field
$\widetilde A$ on ${\cal C}$ setting
\begin{equation}\label{Atilde}
\widetilde
A_J=\frac{1}{2}(\check{A}_{\pi(J)}+J\circ\check{A}_{\pi(J)}\circ J).
\end{equation}

\begin{lemma}\label{h-v} If $J\in{\cal C}$ and $X$ is a vector field near the point
$p=\pi(J)$, then
\begin{itemize}
\item[$(i)$]~~ $[X^h,\widetilde A]_J=(\widetilde{\nabla_X
A})_J$\ ,\\[6pt]
\item[$(ii)$]~~$[X^h,\Phi_k\widetilde A]_J=\Phi_k(\widetilde{\nabla_X A})_J,\quad
k=1,2$\ ,\\[6pt]
\item[$(iii)$]~~$[\Phi_kX^h,\widetilde A]_J=(\widetilde{\nabla_{JX} A})_J-(\widetilde
A_JX)_J^h$\ ,\\[6pt]
\item[$(iv)$]~~$[\Phi_kX^h,\Phi_k\widetilde A]_J=\Phi_k(\widetilde{\nabla_{JX}
A})_J-(\Phi_k(\widetilde
A_J)X)_J^h\ .$\\[6pt]
\end{itemize}
\end{lemma}

\begin{proof}
Note first that by (\ref{h-lift})
\begin{equation}\label{aux}
[X^h,\frac{\partial}{\partial y^{\beta\gamma}}]_J=0,\quad
X^h_J=\sum_{\alpha=1}^{2n+1}X^{\alpha}(p)(\frac{\partial}{\partial
\widetilde x^{\alpha}})_J.
\end{equation}
Let $AE_{\alpha}=\sum_{\beta=1}^{2n+1}a^{\alpha\beta}E_{\beta}$,
$A^{\ast}E_{\alpha}=\sum_{\beta=1}^{2n+1}a^{\ast\,\alpha\beta}E_{\beta}$,
$\check{A}E_{\alpha}=\sum_{\beta=1}^{2n+1}{\check
a}^{\alpha\beta}E_{\beta}$. Then
\begin{equation}\label{Atil}
\widetilde A=\sum_{\alpha,\beta=1}^{2n+1}\widetilde
a^{\alpha\beta}\frac{\partial}{\partial y^{\alpha\beta}},
\end{equation}
where
\begin{equation}\label{tilde a}
\widetilde a^{\alpha\beta}=\frac{1}{2}[{\check
a}^{\alpha\beta}\circ\pi+\sum_{\mu,\nu=1}^{2n+1}y^{\alpha\mu}({\check
a}^{\mu\nu}\circ\pi)y^{\nu\beta}],\quad {\check
a}^{\alpha\beta}=\frac{1}{2}(a^{\alpha\beta}-a^{\ast\,\alpha\beta}).
\end{equation}
In view of (\ref{aux}), it follows  that
\begin{equation}\label{br}
[X^h,\widetilde
A]_J=\frac{1}{2}\sum_{\alpha,\beta=1}^{2n+1}\{X_p({\check
a}^{\alpha\beta})+\sum_{\mu,\nu=1}^{2n+1}y^{\alpha\mu}(J)X_p({\check
a}^{\mu\nu})y^{\nu\beta}(J)\}(\frac{\partial}{\partial
y^{\alpha\beta}})_J.
\end{equation}
On the other hand, we have
$(\nabla_{X_p}\check{A})(E_{\alpha})=\sum_{\beta=1}^{2n+1}X_p(\check{a}^{\alpha\beta})(E_{\beta})_p$.
Note also that
$$\omega((\nabla_{X}A^{\ast})(Y),Z)=\omega(Y,(\nabla_{X}A)(Z))$$
for every $Y,Z\in{\cal D}$ where $(\nabla_{X}A^{\ast})(Y)\in{\cal
D}_p$. Moreover, $(\nabla_{X}A^{\ast})(\xi)=0$. Thus
$(\nabla_{X}A)^{\ast}=\nabla_{X}A^{\ast}$, hence
$(\nabla_{X}A){\check{}}=\nabla_{X}\check{A}$. It follows that
$$
(\widetilde{\nabla_X
A})_J=\frac{1}{2}(\nabla_{X_p}\check{A}-J\circ\nabla_{X_p}\check{A}\circ
J).
$$
Therefore the right-hand side of (\ref{br}) equals
$(\widetilde{\nabla_X A})_J$.

The second formula of the lemma can be proved by similar
computations taking into account that
\begin{equation}\label{PhiA}
\Phi_k\widetilde
A=(-1)^{k+1}\sum_{\alpha,\beta,\gamma=1}^{2n+1}y^{\alpha\beta}\widetilde
a^{\beta\gamma}\frac{\partial}{\partial y^{\alpha\gamma}}.
\end{equation}

Set $X=\sum_{\alpha=1}^{2n+1}f^{\alpha}E_{\alpha}$.  Then
\begin{equation}\label{PhiXh}
\Phi_kX^h=\sum_{\alpha,\beta=1}^{2n+1}(f^{\alpha}\circ\pi)y^{\alpha\beta}E_{\beta}^h.
\end{equation}
This and the first formula of the lemma imply  $(iii)$.

Formula $(iv)$ follows from $(ii)$ and (\ref{PhiXh}).

\end{proof}

\begin{lemma}\label{curv}
For every two vector fields $X,Y$ near the point $p=\pi(J)$ and
every two integers $a,b\geq 0$, we have
$$
[\Phi_k^aX^h,\Phi_k^bY^h]_J=[S^aX,S^bY]^h_J+R_p(J^aX,J^bY)J,\quad
k=1,2.
$$
\end{lemma}
\begin{proof}
This follows from the identities
$(\Phi_k^aX)^h_J=S_{\ast\,p}(J^aX_p)$, $\Phi^a_kX^h\circ
S=(S^aX)^h\circ S$ and formula (\ref{h-bra}).
\end{proof}

Denote by $D=D_t$ the Levi-Civita connection of the metric $G_t$.

\begin{lemma}\label{Dhh}If $X,Y,Z$ are vector
fields on a neighbourhood of the point $p=\pi(J)$, then
$$
\begin{array}{c}
G_t(D_{X^h}Y^h,Z^h)_J=G_t((\nabla_{X}Y)^h,Z^h)_J+[X_p(\alpha(Y))-\alpha_p(\nabla_XY)]\alpha_p(Z)\\[6pt]
+\displaystyle{\frac{1}{2}}[\alpha_p(X)\omega_p(Y,Z)+\alpha_p(Y)\omega_p(X,Z)-\alpha_p(Z)\omega_p(X,Y)]
\end{array}
$$
\end{lemma}
\begin{proof}
Take a local section $S$ of ${\cal C}$ such that $S(p)=J$ and
$\nabla S|_p=0$. Then by the Koszul formula and (\ref{h-bra}) we
have
$$
\begin{array}{c}
2G_t(D_{X^h}Y^h,Z^h)_J=X_p(\omega(Y,SZ))+X(\alpha(Y)\alpha(Z))\\[6pt]
+Y_p(\omega(Z,SX))+Y(\alpha(Z)\alpha(X))-Z_p(\omega(X,SY))-Z(\alpha(X)\alpha(Y))\\[6pt]
+\omega_p(Z,J[X,Y))+\alpha_p(Z)\alpha_p([X,Y])+\omega_p(Y,J[Z,X])+\alpha_p(Y)\alpha_p([Z,X])\\[6pt]
+\omega_p(X,J[Z,Y])+\alpha_p(X)\alpha_p([Z,Y]).
\end{array}
$$
It follows from Lemma~\ref{no} that
$$
X_p(\omega(Y,SZ))=\omega(\nabla_{X_p}Y,JZ)+\omega(Y,J\nabla_{X_p}Z)
$$
in view of the identity $\nabla S|_p=0$. Moreover, we have
$$
\omega(Z,J[X,Y])=\omega(Z,J\nabla_{X}Y)-\omega(Z,J\nabla_{Y}X)
$$
by (\ref{tor}) and (\ref{om-xi}). Also,
$$
\alpha(Z)\alpha([X,Y])=-\alpha(Z)\omega(X,Y)+[X(\alpha(Y))-Y(\alpha(X))]\alpha(Z)
$$
since $\omega=d\alpha$.

 These identities easily imply the lemma.
\end{proof}

\medskip

\noindent {\bf Notation}. Let $A_1,\dots,A_{n^2+n}$ be sections of
$End({\cal D})$ such that $A_1(p),\dots,$ $A_{n^2+n}(p)$ is a
basis of the vertical space ${\cal V}_J\subset End(T_pM)$ and
$\nabla A_{\varepsilon}|_p=0$, $\varepsilon=1,\dots.,n^2+n$. Then
the vector fields $\widetilde A_{\varepsilon}$ constitute a frame
of the vertical bundle in a neighbourhood of $J$.

\medskip

The Koszul formula and Lemma~\ref{h-v} $(i)$ imply that
$(D_{\widetilde A_{\varepsilon}}\widetilde A_{\delta})_J$ is
orthogonal to every horizontal vector $X^h_J$, $X\in T_pM$. Thus we
have the following
\begin{lemma}\label{totgeo}
The fibres of the bundle $\pi:{\cal C}\to M$ are totally geodesic
submanifolds.
\end{lemma}
\begin{lemma}\label{Dhv} If $X,Y$ are vector
fields on a neighbourhood of the point $p=\pi(J)$ and $V$ is a
vertical vector field in a neighbourhood of $J$, then
\begin{equation}\label{VD-h}
G_t(D_{X^h}Y^h,V)_J=\displaystyle{\frac{1}{2}}[-
\omega(X_p,V_JY_p)+G_t(R_p(X,Y)J,V_J)],
\end{equation}
\begin{equation}\label{D-vh}
D_{V}X^h={\cal H}D_{X^h}V,\quad
G_t(D_{V}X^h,Y^h)_J=-G_t(D_{X^h}Y^h,V)_J,
\end{equation}
where ${\cal H}$ means "the horizontal component".
\end{lemma}
\begin{proof}
The Koszul formula, Lemma~\ref{h-v} $(i)$ and identity (\ref{h-bra})
imply
$$
2G_t(D_{X^h}Y^h,\widetilde A_{\varepsilon})_J=-(\widetilde
A_{\varepsilon})_J(G_t(X^h,Y^h))+G_t(R_p(X,Y)J,\widetilde
A_{\varepsilon}).
$$
Let $\gamma$ be a curve in the fibre of ${\cal C}$ through the point
$J$ such that $\gamma(0)=J$ and $\stackrel{\bf
.}{\gamma}(0)=(A_{\varepsilon})_J$. Then
$$
\begin{array}{c}
(\widetilde
A_{\varepsilon})_J(G_t(X^h,Y^h))=
\displaystyle{\frac{d}{dt}}(\omega(X_p,\gamma(t)Y_p)+\alpha(X_p)\alpha(Y_p))|_{t=0}\\[8pt]
= \omega(X_p,(A_{\varepsilon})_JY_p).
\end{array}
$$
This proves the first formula of the lemma.

By Lemma~\ref{totgeo}, $D_{V}X^h$ is orthogonal to every vertical
vector field, thus it is horizontal. Moreover, $[V,X^h]$ is a
vertical vector field, hence $D_{V}X^h={\cal H}D_{X^h}V$. We also
have
$$G_t(D_{V}X^h,Y^h)=G_t([V,X^h]+D_{X^h}V,\,Y^h)=-G_t(V,D_{X^h}Y^h).$$
\end{proof}

Lemmas~\ref{Dhh} and \ref{Dhv} imply the following.
\begin{cotmb}\label{D-xi-h} Let $X,Y\in T_pM$ and $V,W\in{\cal V}_J$.
Then
$$
\begin{array}{c}
G_t(D_{X^h}\xi^h,Y^h)_J=\displaystyle{\frac{1}{2}}\omega(X,Y), \quad
G_t(D_{V}\xi^h,W)=0\\[8pt]
G_t(D_{X^h}\xi^h,V)_J=G_t(D_{V}\xi^h,X^h)_J=\displaystyle{\frac{1}{2}}G_t(R(X,\xi_p)J,V).
\end{array}
$$
\end{cotmb}
\begin{cotmb}\label{d-eta}
$$
d\eta_t(X^h,Y^h)=\omega(X,Y),\quad
d\eta_t(X^h,V)=d\eta_t(V,W)=0,\quad \delta\eta_t=0.
$$
\end{cotmb}
\begin{cotmb}
Every integral curve of $\xi^h$ is a geodesic in $({\cal C},G_t)$.
\end{cotmb}
\begin{cotmb}\label{Killing}
The vector field $\xi^h$ on $({\cal C},G_t)$ is Killing if and only
if $R(X,\xi)Y=0$ for every $X,Y\in TM$.
\end{cotmb}
\begin{proof}
By Corollary~\ref{D-xi-h}, $\xi^h$ is Killing if and only if
$R(X,\xi)J=0$ for every $J\in{\cal C}$ and $X\in T_{\pi(J)}M$. Fix
a tangent vector $X\in T_pM$ and suppose that $R(X,\xi)J=0$ for
every $J$ in the fibre ${\cal C}_p$ of the bundle ${\cal C}$. Let
$e_1,\dots,e_{2n}$ be a symplectic basis of ${\cal D}_p$ and let
$J$ be the complex structure of ${\cal D}_p$ corresponding to this
basis, $Je_{i}=e_{i+n}$, $i=1,\dots,n$. Then $J\in {\cal C}_p$ and
the identity $R(X,\xi)J=0$ implies
\begin{equation}\label{k1}
\omega(R(X,\xi)e_{i+n},e_{k})+\omega(R(X,\xi)e_{i},e_{k+n})=0,\quad
i,k=1,\dots,n,
\end{equation}
\begin{equation}\label{k2}
\omega(R(X,\xi)e_{i+n},e_{k+n})-\omega(R(X,\xi)e_{i},e_{k})=0.
\end{equation}
For $\lambda\in{\Bbb R}$,  as in \cite{V86}, consider the symplectic
basis $e_i'=e_i, e_{i+n}'=\lambda e_i+e_{i+n}$. Applying (\ref{k1})
for this basis, we get $\omega(R(X,\xi)e_i,e_k)=0$, hence
$\omega(R(X,\xi)e_{i+n},e_{k+n})=0$ by (\ref{k2}). It follows from
the identities
$$
\omega(R(X,\xi)e_i,e_k)=\omega(R(X,\xi)e_{i+n},e_{k+n})=0,
$$
Lemma~\ref{Bi} $(i)$, and identity (\ref{k1}) that
$\omega(R(X,\xi)Z,Z)=0$ for every $Z\in{\cal D}_p$. In view of
Lemma~\ref{Bi} (i), polarization of the latter identity gives
$\omega(R(X,\xi)Y,Z)=0$ for $Y,Z\in{\cal D}_p$. Therefore
$R(X,\xi)Y=0$  for $Y\in{\cal D}_p$. For $Y=\xi$ this is obvious.

Conversely, if $R(X,\xi)Y=0$ for every $X,Y$, we have clearly
$R(X,\xi)J=0$, so $\xi^h$ is Killing.
\end{proof}

\begin{prop}\label{N1}
Let $J\in{\cal C}$, $X,Y\in T_{\pi(J)}M$,
$V,W\in{\cal V}_J$. Then
$$
\begin{array}{c}
N^{(1)}_k(X^h,Y^h)_J=-R(X,Y)J+R(JX,JY)J\\[6pt]
-(-1)^{k+1}J(R(JX,Y)J+R(X,JY)J)\ ,\\[6pt]
N^{(1)}_k(X^h,V)=[1+(-1)^{k}](JVX)^h_J\ ,\quad N^{(1)}_k(V,W)=0\ .
\end{array}
$$
\end{prop}
\begin{proof}
Extending $X,Y$ to vector fields in a neighbourhood of the point
$p=\pi(J)$ and taking into account Lemma~\ref{curv},
Corollary~\ref{d-eta}, and identity (\ref{tor}), we easily see that
$$
\begin{array}{c}
{\cal H}N^{(1)}_k(X^h,Y^h)_J=\\[6pt]
\big(-S_p(\nabla_{X}S)_p(Y)+S_p(\nabla_{Y}S)_p(X)+(\nabla_{SX}S)_p(Y)-(\nabla_{SY}S)_p(X)\big)^h_J\\[6pt]
-d\alpha(SX_p,SY_p)\xi^h_J+\omega(X_p,Y_p)\xi^h_J,
\end{array}
$$
$S$ being defined in Section 3. We have $d\alpha(SX_p,SY_p)=\omega(JX_p,JY_p)=\omega(X_p,Y_p)$. Hence
$${\cal H}N^{(1)}_k(X^h,Y^h)_J=0$$
since $\nabla S|_p=0$. Now the first formula of the lemma follows
from Lemma~\ref{curv} and the fact that $\Phi_k^2=-Id$ on the
vertical spaces.

Let $A$ be a section of $End({\cal D})$ such that $A_p=V$. Denote by
$\widetilde A$ the vertical vector field on ${\cal C}$ defined by
(\ref{Atilde}). Then
$$
N^{(1)}_k(X^h,V)=N^{(1)}_k(X^h,\widetilde
A)_J=[1+(-1)^{k}](J\widetilde A_JX)^h_J
$$
by Lemma~\ref{h-v} and Corollary~\ref{d-eta}.

Corollary~\ref{d-eta} and the fact that $\Phi_k$ is a complex
structure on the fibres of ${\cal C}$ imply $N^{(1)}_k(V,W)=0$.
\end{proof}

Now we set
$$
R_{\cal D}(X,Y,Z,T)=\omega(R(X,Y)Z,T)~\mbox{for}~X,Y,Z,T\in{\cal D}.
$$
Note that this covariant $4$-tensor satisfies the identities $(i)$,
$(ii)$, $(iii)$ in Section 2.2

\begin{tw}\label{nor} $(i)$ The almost contact metric structure $(\Phi_1,\xi^h,G_t)$ is normal
if and only if $R(X,\xi)Y=0$ for every $X,Y\in TM$ and the tensor
$R_{\cal D}$ is of Ricci type.

\noindent $(ii)$ The almost contact metric structure
$(\Phi_1,\xi^h,G_t)$ is never normal.
\end{tw}
\begin{proof}
By Proposition~\ref{N1}, $(\Phi_1,\xi^h,G_t)$ is a normal structure
if and only if $N^{(1)}_1(X^h,Y^h)_J=0$ for every$J\in{\cal C}$ and
$X,Y\in T_{\pi(J)}M$.

Note first that  $N^{(1)}_1(X^h,Y^h)_J$, with $J$, $X,Y$ fixed, is
a linear operator on $T_{\pi(J)}M$ whose value at $\xi_{\pi(J)}$
is zero.

Take a point $p\in M$. According to Proposition~\ref{N1},
$N^{(1)}_1(X^h,\xi^h)_J=0$ for $J\in{\cal C}_p$ and $X\in T_pM$,
if and only if for every $Z,T\in{\cal D}_p$ and $J\in{\cal C}_p$
\begin{equation}\label{Nor-2}
\begin{array}{c}
\omega(R(X,\xi)JZ,T)+\omega(R(X,\xi)Z,JT)\\[6pt]
-\omega(JX,\xi)JZ,JT)+\omega(R(JX,\xi)Z,T)=0.
\end{array}
\end{equation}
This obviously holds for $X=\xi$, so assume that $X\in{\cal D}_p$.
Let $e_{\alpha}$, $\alpha=1,\dots,2n$, be a symplectic basis of
${\cal D}_p$. It is convenient to set
$$
R_{\alpha,\beta,\gamma}=\omega(R(e_{\alpha},\xi)e_{\beta},e_{\gamma}).
$$
Then
\begin{equation}\label{abg}
R_{\alpha,\beta,\gamma}=R_{\alpha,\gamma,\beta}=R_{\beta,\alpha,\gamma}
\end{equation}
For $\lambda\in{\Bbb R}$,  consider the symplectic basis $e_i'=e_i,
e_{i+n}'=\lambda e_i+e_{i+n}$. Applying (\ref{Nor-2}) for the
complex structure $J'$ corresponding to this basis and $X=e_i',
Z=e_j', T=e_k'$, we obtain
\begin{equation}\label{e'}
\begin{array}{c}
R_{i,j,k}=0,\\[6pt]
 R_{i,j,k+n}+R_{i,j+n,k}+R_{i+n,j,k}=0,\>
R_{i,j+n,k+n}+R_{i+n,j,k+n}+R_{i+n,j+n,k}=0.
\end{array}
\end{equation}
Consider also the symplectic basis $e_i''=e_i+\lambda e_{i+n},
e_{i+n}''=e_{i+n}$. Setting $X=e_i'', Z=e_j'', T=e_k''$ in
(\ref{Nor-2}) and taking into account (\ref{e'}), we get
\begin{equation}\label{e''}
R_{i+n,j+n,k+n}=0,\quad R_{i+n,j+n,k}=0,\quad R_{i,j,k+n}=0.
\end{equation}
Now $R_{i,j+n,k}=R_{i,k,j+n}=0$ and
$R_{i,j+n,k+n}=R_{j+n,i,k+n}=R_{j+n,k+n,i}=0$ by (\ref{abg}) and
(\ref{e''}). Similarly, identities (\ref{abg}) and (\ref{e''})
imply $R_{i+n,j,k}=R_{i+n,j,k+n}=0$. It follows that
$R_{\alpha,\beta,\gamma}=0$ for every
$\alpha,\beta,\gamma=1,\dots,2n$. Therefore
$\omega(R(X,\xi)Z,T)=0$ for every $Z,T\in{\cal D}$. This implies
$R(X,\xi)Z=0$ since $R(X,\xi)Z\in{\cal D}$ and $\omega$ is
non-degenerate on ${\cal D}$.

Conversely, if $R(X,\xi)Z=0$ for every $Z$, identity (\ref{Nor-2})
is obviously satisfied, so  $N^{(1)}_1(X^h,\xi^h)=0$, $X\in TM$.

Next, we discuss the identity $N^{(1)}_1(X^h,Y^h)_J=0$ for $X,Y\in
{\cal D}_{\pi(J)}$. It is convenient to introduce the operator
$J^{-}=\frac{1}{2}(Id+iJ)$ on the complexification ${\cal D}^{\Bbb
C}$ of ${\cal D}$ and to extend $\omega$ to ${\cal D}^{\Bbb C}$ by
complex bilinearity. Then, taking into account Proposition~\ref{N1},
it is easy to check that $\omega(N^{(1)}_1(X^h,Y^h)_J(Z),T)=0$ for
$X,Y,Z,T\in {\cal D}_{\pi(J)}$, if and only if $R_{\cal
D}(J^{-}X,J^{-}Y,,J^{-}Z,J^{-}T)=0$. The latter condition is
equivalent to $R_{\cal D}$ being of Ricci type by \cite[Lemma
1.1]{Vezz}.

This proves the first part of the theorem.

In order to see that the structure $(\Phi_1,\xi^h,G_t)$ is not
normal, we fix a point $p\in M$ and take a symplectic basis
$E_1,\dots,E_{2n}$ of ${\cal D}_p$.  Let $J$ be the complex
structure on ${\cal D}_p$ for which $JE_{i}=E_{i+n}$,
$i=1,\dots,n$. Define a vertical vector $V_{1,2}$ of ${\cal C}$ at
$J$ by formula (\ref{Vij}). Then $V_{1,2}E_1=E_{2+n}$, so
$N^{(1)}_2(E_1^h,V_{1,2})_J=-2(E_2^h)_J\neq 0$, by
Proposition~\ref{N1}.
\end{proof}

\smallskip

\noindent {\it Remark 5}.  Note that, by Corollary~\ref{Killing},
the condition $R(X,\xi)Y=0$ for every $X,Y$ means that the vector
field $\xi^h$ on $({\cal C},G_t)$ is Killing.

\section{Almost $CR$-structures on contact twistor space}

In this section we shall show that the almost $CR$-structure $({\cal
E},\Phi_2|{\cal E})$ is not integrable, reproving in passing the
integrability result of \cite{Vezz} for $({\cal E},\Phi_1|{\cal
E})$; here ${\cal E}$ is the bundle over ${\cal C}$ whose fibre at a
point $J\in{\cal C}$ is the space ${\cal V}_J\oplus
\{X^h_J:~X\in{\cal D}_{\pi(J)}\}=({\Bbb R}\xi_{\pi(J)})^{\perp}$,
the orthogonal complement being with respect to the metric $G_t$.

Recall that an almost  Cauchy-Riemann ($CR$) structure  on a
manifold $N$ is a pair $({\cal E},\Phi)$ of a subbundle ${\cal E}$
of the tangent bundle $TN$ and an almost complex structure $\Phi$ of
the bundle ${\cal E}$. For any two sections $X,Y$ of ${\cal E}$, the
value of $[X,Y] \> mod\, {\cal E}$ at a point $p\in N$ depends only
on the values of $X$ and $Y$ at $p$, so we have a skew-symmetric
bilinear form ${\cal L}: {\cal E}\times {\cal E}\to TN/{\cal E}$
defined by ${\cal L}(X,Y)=[X,Y] \> mod\, {\cal E}$ and called the
Levi form of the $CR$-structure $({\cal E},\Phi)$. If the Levi form
is $\Phi$-invariant, we can define the Nijenhuis tensor of the
$CR$-structure $({\cal E},\Phi)$ by
$$
N^{\it CR}(X,Y)=-[X,Y]+[\Phi X,\Phi Y]-\Phi([\Phi X,Y]+[X,\Phi
Y]).
$$
The value of this tensor at a point $p\in N$ lies in ${\cal E}$ and
depends only on the values of the sections $X,Y$ at $p$. An almost
$CR$-structure is said to be integrable if its Levi form is
$\Phi$-invariant and the Nijenhuis tensor vanishes. Let ${\cal
E}^{\Bbb C}={\cal E}^{1,0}\oplus {\cal E}^{0,1}$ be the
decomposition of the complexification of ${\cal E}$ into $(1,0)$ and
$(0,1)$ parts with respect to $\Phi$. If the $CR$-structure $({\cal
E},\Phi)$ is integrable, then the bundle ${\cal E}^{1,0}$ satisfies
the following two conditions:
$$
{\cal E}^{1,0}\cap \overline{{\cal E}^{1,0}}=0, \> [\Gamma({\cal
E}^{1,0}),\Gamma({\cal E}^{1,0})]\subset\Gamma({\cal E}^{1,0})
$$
where $\Gamma({\cal E}^{1,0})$ stands for the space of smooth
sections of ${\cal E}^{1,0}$. Conversely, suppose we are given a
complex subbundle $B$ of the complexified tangent bundle $T^{\Bbb
C}N$ such that $B\cap\overline B=0$ and
$[\Gamma(B),\Gamma(B)]\subset\Gamma(B)$ (many authors call a bundle
with these properties "$CR$-structure"). Set ${\cal E}=\{X\in TN:
X=Z+\bar Z\> \textrm{ for some (unique)} \> Z\in B\}$ and put $\Phi
X=-2{\it Im}Z$ for $X\in {\cal E}$. Then $({\cal E},\Phi)$ is an
integrable $CR$-structure such that ${\cal E}^{1,0}=B$.

\smallskip

Let ${\cal L}_k$ be the Levi form of the almost $CR$-structure
$({\cal E},\Phi_k)$, $k=1,2$.

\begin{lemma}\label{Levi}
Let $J\in{\cal C}$, $X,Y\in {\cal D}_{\pi(J)}$ and $U,V\in{\cal
V}_J$. Then
$$ {\cal L}_{k}(X^h_J,Y^h_J)=-\omega(X,Y), \quad {\cal L}_{k}(U,V)=0,
\quad {\cal L}_{k}(X^h,V)=0.$$
\end{lemma}

\begin{proof} Extend $X$ and $Y$ to sections of ${\cal D}$ near the point
$p=\pi(J)$ such that $\nabla X|_p=\nabla Y|_p=0$. Then, by
Lemma~\ref{Dhh},
$$
{\cal
L}_k(X^h_J,Y^h_J)=G_t([X^h,Y^h],\xi^h)_J=G_t(D_{X^h}Y^h-D_{Y^h}X^h,\xi^h)_J=-\omega_p(X,Y).
$$
Extend $U,V$ to vertical vector fields. The vector fields $[U,V]$
and $[X^h,V]$ are vertical, hence $ {\cal L}_k(U,V)={\cal
L}_k(X^h,V)=0$.
\end{proof}

\begin{cotmb}
The Levy form ${\cal L}_k$ is $\Phi_k$-invariant, $k=1,2$.
\end{cotmb}

Denote the Nijenhuis tensor of the $CR$-structure $({\cal
E},\Phi_k)$ by $N^{\it CR}_k$.

\begin{prop}\label{NCR}
If $P,Q\in {\cal E}$, then $ N^{\it CR}_k(P,Q)=N^{(1)}_k(P,Q)$.
\end{prop}
\begin{proof} Let $J\in{\cal C}$ and let $X$, $Y$ be sections of ${\cal D}$ near the point
$p=\pi(J)$. Using Lemma~\ref{curv} and identity (\ref{tor}), we
easily see that
$$
\begin{array}{c}
{\cal H}N^{\it CR}_k(X^h,Y^h)_J=\\[6pt]
\big(-S_p(\nabla_{X}S)_p(Y)+S_p(\nabla_{Y}S)_p(X)+(\nabla_{SX}S)_p(Y)-(\nabla_{SY}S)_p(X)\big)^h_J\\[6pt]
+ \big[\omega(X_p,Y_p)-\omega(SX_p,SY_p)\big]\xi^h_J.
\end{array}
$$
Therefore,
$$
{\cal H}N^{\it CR}_k(X^h,Y^h)_J=0={\cal H}N^{(1)}_k(X^h,Y^h)_J
$$
by Proposition~\ref{N1}. Lemma~\ref{curv} implies also that ${\cal
V}N^{\it CR}_k(X^h,Y^h)_J={\cal V}N^{(1)}_k(X^h,Y^h)_J$. It
follows from Lemma~\ref{h-v} and Proposition~\ref{N1} that $N^{\it
CR}_k(X^h,U)_J=N^{(1)}_k(X^h,U)_J$ for every $U\in{\cal V}_J$.
Finally, if $U,V\in {\cal V}_J$, then $N^{\it
CR}_k(P,Q)=N^{(1)}_k(P,Q)=0$ since $\Phi_k$ is a complex structure
on the fibres of ${\cal C}$.
\end{proof}

Proposition~\ref{NCR} and the proof of Theorem~\ref{nor} give the
following.
\begin{tw}\label{intCR} $(i)$ {\rm{(\cite{Vezz})}} The almost $CR$ structure $({\cal E},\Phi_1)$
is integrable if and only if  the tensor $R_{\cal D}$ is of Ricci
type.

\noindent $(ii)$ The almost $CR$ structure $({\cal E},\Phi_2)$ is
never integrable.
\end{tw}

\section{Examples}

\noindent {\it Example 1}. Let $E_1,E_2,E_3$ be left-invariant
vector fields on the group $SO(3)$ such that
$$
[E_1,E_2]=E_3,\quad [E_2,E_3]=E_1,\quad [E_3,E_1]=E_2.
$$
If $E_1^{\ast},E_2^{\ast},E_3^{\ast}$ is the dual frame, set
$\alpha=-E_3^{\ast}$. Then
$$
d\alpha(E_1,E_2)=1,\quad d\alpha(E_1,E_3)=d\alpha(E_2,E_3)=0.
$$
Thus $\alpha$ is a contact form on $SO(3)$ with contact distribution
${\cal D}=span\{E_1,E_2\}$ and Reeb vector field $\xi=-E_3$.

It is easy to check that for every contact connection the tensor
$R_{\cal D}$ is of Ricci type.

A simple example of a contact connection $\nabla$ can be found
setting
$$
\nabla_{E_i}E_1=a_iE_1+b_iE_2,\quad
\nabla_{E_i}E_2=c_iE_1+d_iE_2,\quad i=1,2,
$$
where $a_i, b_i, c_i, d_i\in{\Bbb R}$. Then the identities
$$(\nabla_{E_1}d\alpha)(E_1,E_2)=(\nabla_{E_2}d\alpha)(E_1,E_2)=0$$
are equivalent to
\begin{equation}\label{cc1}
a_1=-d_1,\quad a_2=-d_2.
\end{equation}
We have
$$
[E_1,E_2]=\nabla_{E_1}E_2-\nabla_{E_2}E_1+\alpha\big([E_1,E_2]\big)\xi,
\quad i,j=1,2,
$$
if and only if
\begin{equation}\label{cc2}
a_2=c_1,\quad b_2=d_1.
\end{equation}
Suppose that identities (\ref{cc1}) and (\ref{cc2}) are satisfied
and set $\nabla_{\xi}E_i=[\xi,E_i]$, $i=1,2$, and $\nabla\xi=0$.
Then $\nabla$ is a contact connection for which
$$
\begin{array}{c}
R(E_1,\xi)E_1=(2a_2+b_1)E_1+3d_1E_2,\\[6pt]
R(E_1,\xi)E_2=(c_2+2d_1)E_1-(b_1-2d_2)E_2\\[6pt]
R(E_2,\xi)E_2=3d_2E_1-(2d_1+c_2)E_2.
\end{array}
$$
Thus $R(X,\xi)Y=0$ for $X,Y\in{\cal D}$ if and only if
$a_i=b_i=c_i=d_i=0$, $i=1,2$.

\medskip

\noindent {\it Example 2}. Let $G$ be the simply connected
$5$-dimensional Lie group with Lie algebra generated by
left-invariant vector fields $E_1,\dots,E_5$ whose non-zero Lie
brackets are
$$
[E_2,E_3]=E_1,\quad [E_2,E_5]=E_2,\quad [E_3,E_5]=-E_3,\quad
[E_4,E_5]=E_1.
$$
One can easily see that the group $G$ is solvable. If
$E_1^{\ast},\dots,E_5^{\ast}$ is the dual frame of left-invariant
$1$-forms, $E_1^{\ast}$ is a contact form on $G$. It is shown in
\cite{DF} that $G$ admits a lattice $\Gamma$ such that the
quotient $G/\Gamma$ is compact (recall that such a lattice is
called uniform). In fact, $G$ is one of the groups in the list,
obtained in \cite{DF}, of all simply connected solvable
$5$-dimensional Lie groups admitting a left-invariant contact form
and a uniform lattice.

Let $s\neq 0$ be a real number. Then $\alpha=sE_1^{\ast}+E_4^{\ast}$
is a contact form on $G$ with contact distribution  ${\cal
D}=span\{E_2,E_3,E_4-\displaystyle{\frac{1}{s}}E_1,A_5\}$ and Reeb
field $\xi=\displaystyle{\frac{1}{s}}E_1$. It is convenient to set
$$
A_1=E_2,\quad A_2=E_3,\quad A_3=E_4-\displaystyle{\frac{1}{s}}E_1,
\quad A_4=E_5, \quad A_5=\xi.
$$
We have the following table for the non-zero Lie brackets of the
vector fields $A_1,\dots,A_5$
$$
[A_1,A_2]=sA_5,\quad [A_1,A_4]=A_1,\quad [A_2,A_4]=-A_2,\quad
[A_3,A_4]=sA_5.
$$

The only non-zero values of the form $\omega=d\alpha$ are
$$
\omega(A_1,A_2)=\omega(A_3,A_4)=-s.
$$

\smallskip

 Let $\nabla'$ be the connection on $G$ for which
$$
\begin{array}{c}
\nabla'_{A_i}A_j=\frac{1}{2}([A_i,A_j]-\alpha([A_i,A_j])\xi),\quad
i,j=1,\dots,4,\\[6pt]
\nabla'_{A_5}A_i=[A_5,A_i],\quad i=1,\dots,4,\quad \nabla' A_5=0.
\end{array}
$$
This connection satisfies all conditions in the definition of a
contact connection except the condition (\ref{nda}) (for example,
$(\nabla'_{A_1}\omega)(A_2,A_4)=-\frac{s}{2}\neq 0$). To get a
contact connection we follow the procedure used in the proof of
\cite[Theorem 2.5]{Vezz}. If ${\cal N}$ is the tensor on ${\cal D}$
defined by
$$
\omega({\cal N}(X,Y),Z)=(\nabla'_{X}\omega)(Y,Z),\quad X,Y,Z\in{\cal
D},
$$
set
$$
\begin{array}{c}
\widetilde\nabla_{X}{Y}=\nabla'_{X}{Y}+\displaystyle{\frac{1}{3}}{\cal
N}(X,Y)+\displaystyle{\frac{1}{3}}{\cal N}(Y,X)~~\mbox{for}~~X,Y\in{\cal D},\\[8pt]
\widetilde\nabla_{\xi}X=[\xi,X]~~\mbox{for}~~X\in{\cal D}, \quad
\widetilde\nabla\xi=0.
\end{array}
$$
Then $\widetilde\nabla$ is a contact connection.

In our case, the non-zero values of ${\cal N}$ are
$$
\begin{array}{c}
{\cal N}(A_1,A_2)={\cal
N}(A_2,A_1)=\displaystyle{\frac{1}{2}}A_3,\quad {\cal
N}(A_1,A_4)=-\displaystyle{\frac{1}{2}}A_1,\\[8pt]
{\cal N}(A_2,A_4)=\displaystyle{\frac{1}{2}}A_2.
\end{array}
$$
Thus we have the following table for the connection
$\widetilde\nabla$
$$
\widetilde\nabla_{A_1}A_2=\frac{1}{3}A_3,\quad
\widetilde\nabla_{A_1}A_4=\frac{1}{3}A_1,\quad
\widetilde\nabla_{A_2}A_1=\frac{1}{3}A_3,\quad
\widetilde\nabla_{A_2}A_4=-\frac{1}{3}A_2
$$
$$
\widetilde\nabla_{A_4}E_1=-\frac{2}{3}A_1,\quad
\widetilde\nabla_{A_4}E_2=\frac{2}{3}A_2
$$
and all other $\widetilde\nabla_{A_i}A_j$ vanish.

The tensor $R_{\cal D}$ for this connection is not of Ricci type
(identity (\ref{red}) with $P=\sigma$ is not satisfied for
$X=A_1,Y=A_4,Z=A_4,U=A_4$). According to {\cite[Theorem 2.5]{Vezz}
every contact connection $\widetilde\nabla$ is of the form
$$
\nabla_{X}Y=\widetilde\nabla_{X}Y+S(X,Y)
$$
where  $S(X,Y)$ is a tensor with the following properties:
\begin{itemize}
\item[(a)]\quad $S(X,Y)$ takes its values in
${\mathcal D}$;
\item[(b)]\quad $S(X,\xi)=S(\xi,X)=0$ for every $X$;
\item[(c)]\quad $S(X,Y)=S(Y,X)$;
\item[(d)]\quad $\omega(S(X,Y),Z)$ is a symmetric $3$-tensor.
\end{itemize}
In order to find a tensor $S$ such that the curvature $R_{\cal D}$
of the corresponding connection $\nabla$ is of Ricci type, we have
used a computer computation. This suggested the following simple
choice of $S$
$$
\begin{array}{c}
S(A_1,A_2)=S(A_2,A_1)=-\displaystyle{\frac{1}{3}}A_3,\quad
S(A_1,A_4)=S(A_4,A_1)=-\displaystyle{\frac{1}{3}}A_1,\\[6pt]
S(A_2,A_4)=S(A_4,A_2)=\displaystyle{\frac{1}{3}}A_2,
\end{array}
$$
all other $S(A_i,A_j)=0$. Under that choice of $S$ we get a flat
contact connection $\nabla$ for which
$$
\nabla_{A_4}A_1=-A_1,\quad \nabla_{A_4}A_2=A_2,\quad
\nabla_{A_i}A_j=0~\,\mbox{for}~ (i,j)\neq (4,1), (4,2).
$$

\medskip

\noindent {\it Example 3}. As in \cite[Example 2]{Dia}, consider
a $5$-dimensional Lie group $G$ with Lie algebra generated by
left-invariant vector fields $E_0,E_1,\dots,E_4$ whose non-zero
Lie brackets are
$$
\begin{array}{c}
[E_0,E_1]=-E_1,\quad [E_0,E_2]=E_2,\quad [E_1,E_2]=E_3,\\[6pt]
[E_1,E_4]=-E_1,\quad  [E_3,E_4]=-E_3
\end{array}
$$
Let $s\neq 0$ be a real number. If
$E_0^{\ast},E_1^{\ast},\dots,E_4^{\ast}$ is the dual frame of
left-invariant $1$-forms, $\alpha=E_3^{\ast}+sE_0^{\ast}$ is a
contact form on $G$ with contact distribution  ${\cal
D}=span\{E_1,E_2,E_4,sE_3-E_0\}$  and Reeb field
$\xi=\displaystyle{\frac{1}{s}}E_0$. Set
$$
A_1=E_1,\quad A_2=E_2,\quad A_3=E_4, \quad A_4=sE_3-E_0, \quad
A_5=\xi.
$$
Considerations similar to that in the preceding example lead to
the following non-flat contact connection $\nabla$ whose curvature
$R$ satisfies the identity $R(X,\xi)Y=0$ for every $X,Y$ and the
tensor $R_{\cal D}$ is of Ricci type (thus, the almost contact
metric structure $(\Phi_1,\xi^h,G_t)$ is normal):
$$
\nabla_{A_1}A_2=\frac{1}{2s}A_4,\quad
\nabla_{A_1}A_3=-\frac{1}{2}A_1,\\[6pt]
$$
$$
\nabla_{A_2}A_1=-\frac{1}{2s}A_4,\quad
\nabla_{A_2}A_3=-\frac{1}{2}A_2,\\[6pt]
$$
$$
\nabla_{A_3}A_1=\frac{1}{2s}A_1,\quad
\nabla_{A_3}A_2=-\frac{1}{2}A_2,\quad \nabla_{A_3}A_4=-2A_3,\\[6pt]
$$
$$
\nabla_{A_4}A_1=A_1,\> \nabla_{A_4}A_2=-A_2,\>
\nabla_{A_4}A_3=-2A_3-A_4,\>\nabla_{A_4}A_4=8A_3+2A_4,\\[6pt]
$$
$$
\nabla_{A_5}A_1=-\frac{1}{s}A_1,\quad
\nabla_{A_5}A_2=\frac{1}{s}A_2,\\[6pt]
$$
$$
\mbox{all other}~~ \nabla_{A_i}A_j=0.
$$
The tensor $R_{\cal D}$ of the following non-flat contact connection
is of Ricci type, while $R(X,\xi)Y$ is not identically zero, so the
almost $CR$ structure $({\cal E},\Phi_1)$ is integrable but the
almost contact metric structure $(\Phi_1,\xi^h,G_t)$ is not normal.
$$
\nabla_{A_1}A_2=-\frac{1}{3s}A_3+\frac{2}{3s}E_4,\quad
\nabla_{A_1}A_3=-\frac{2}{3}A_1,\quad
\nabla_{A_1}A_4=-\frac{1}{3}E_1,
$$
$$
\nabla_{A_2}A_1=-\frac{1}{3s}A_3-\frac{1}{3s}A_4,\quad
\nabla_{A_2}A_3=-\frac{1}{2}A_2,\quad
\nabla_{A_2}A_4=\frac{1}{3}A_2,
$$
$$
\nabla_{A_3}A_1=\frac{1}{3}A_1,\quad
\nabla_{A_3}A_2=-\frac{1}{3}A_2,\quad
\nabla_{A_3}A_3=-\frac{1}{3}A_3, \nabla_{A_3}A_4=\frac{1}{3}A_3
$$
$$
\nabla_{A_4}A_1=\frac{2}{3}A_1,\quad
\nabla_{A_4}A_2=-\frac{2}{3}A_2,\quad
\nabla_{A_4}A_3=-\frac{2}{3}A_4,
$$
$$
\nabla_{A_5}A_1=-\frac{1}{s}A_1,\quad
\nabla_{A_5}A_2=\frac{1}{s}A_2,
$$
$$
\mbox{all other}~~ \nabla_{A_i}A_j=0.
$$

Finally, note that the Lie algebra $\frak g=span\{E_0,\dots,E_4\}$
of the group $G$ is solvable. It is not unimodular
($Trace\,ad_{E_4}=2$), hence, by a result of Milnor \cite{Mil},
the group $G$ does not possess a discrete subgroup $\Gamma$ such
that the quotient $G/\Gamma$ is compact.

\end{document}